\documentclass[english, 12pt]{amsart}

\usepackage{geometry}
\usepackage[T1]{fontenc}
\usepackage[utf8]{inputenc}
\usepackage{babel}
\usepackage{amsmath}
\usepackage{amssymb}
\usepackage{amsthm}
\usepackage{textcomp}
\usepackage{enumitem}
\usepackage{dsfont}
\usepackage{mathrsfs}
\usepackage{cite}
\usepackage{color}
\usepackage[colorlinks=true,urlcolor=blue,linkcolor=blue]{hyperref}

\newcommand{\R}{\mathbb{R}}

\newcommand{\C}{\mathbb{C}}

\newcommand{\ii}{\mathrm{i}}
\newcommand{\dd}{\,\mathrm{d}}
\newcommand{\ee}{\mathrm{e}}

\renewcommand{\varepsilon}{\epsilon}

\def\build#1_#2^#3{\mathrel{
\mathop{\kern 0pt#1}\limits_{#2}^{#3}}}
\def\td_#1,#2{\mathrel{\mathop{\build\longrightarrow_{#1\rightarrow #2}^{}}}}
\def\dl_#1,#2{\mathrel{\mathop{\build=_{#1\rightarrow #2}^{}}}}
\renewcommand{\Im}{\mathrm{Im}}
\renewcommand{\Re}{\mathrm{Re}}

\usepackage{tikz}

\newcommand{\longrightarroww}[2] {\mathop{\longrightarrow}\limits_{#1}^{#2}}
\def\tdw_#1,#2{\mathrel{\mathop{\build\rightharpoonup_{#1\rightarrow #2}^{}}}}

\newtheorem{definition}{Definition}[section]
\newtheorem{theorem}[definition]{Theorem}
\newtheorem{lemma}[definition]{Lemma}
\newtheorem{proposition}[definition]{Proposition}
\newtheorem{corollary}[definition]{Corollary}
\newtheorem{remark}[definition]{Remark}
\numberwithin{equation}{section}

\title[Infinite-order multisolitons for the Benjamin--Ono equation]{Infinite-order multisoliton solutions to the Benjamin--Ono equation and soliton resolution}
\date{\today}

\author[L.~Gassot]{Louise Gassot}
\address{CNRS and Department of Mathematics, University of Rennes, France}
\email{louise.gassot@cnrs.fr}

\author[P.~Gérard]{Patrick Gérard}
\address{Laboratoire de Mathématiques d'Orsay,  Université Paris-Saclay, Orsay, France} \email{patrick.gerard@universite-paris-saclay.fr}



\begin{document}

\maketitle

\begin{abstract}
We construct a class of infinite-order multisoliton solutions of the Benjamin-Ono equation on the line, for which the initial data exhibits slow spatial decay. We prove that in the long-time asymptotics, such a solution decouples as an infinite superposition of independent soliton solutions with different velocities and no radiation term.
\vskip 0.25cm 
\centerline{\em Dedicated to Yoshio Tsutsumi, on the occasion of his seventieth birthday}
\end{abstract}

\setcounter{tocdepth}{1}
\tableofcontents

\section{Introduction and main result}\label{sec1}

This paper is devoted to the long-time behavior of solutions of the Benjamin--Ono equation,
\begin{equation}\label{BOinit}
\partial_tu-\partial_x|D_x|u+\partial_x(u^2)=0\ ,\ u(0,x)=u_0(x).
\end{equation}
The function $u=u(t,x)$ is real--valued, and $|D_x|$ is the Fourier multiplier associated to the symbol $|\xi |$ acting on functions on the real line. Equation \eqref{BOinit} was introduced in 1967 \cite{Benjamin67} in order to model long, unidirectional internal gravity waves 
in a two-layer fluid with infinite depth, see e.g. the book \cite{KleinS21} for more detail and \cite{Paulsen24} for a recent derivation.

We recall that an important class of  solutions of the Benjamin-Ono equation~\eqref{BOinit} are the traveling wave solutions, 
\[ u(t,x)=u_0(x-ct).\]
These solutions  were characterized in~\cite{AmickToland1991}, and have the form
\begin{equation}\label{eq:Rp}
u(t,x)=R_{p}(x-c_pt),
\quad
R_p(x)=\frac{2 \Im(p)}{|x+p|^2},
\quad
c_p=\frac{1}{\Im(p)},
\end{equation}
where $p\in\C_+=\{z\in\C \;\colon\; \Im(z)>0\}$ is a complex number with positive imaginary part. Note that the $L^2$ norm of such a solution is given by
\[
\|R_p\|_{L^2}^2=4\pi c_p.
\]
In this paper, these solutions will be called soliton solutions.

It has been recently proved in~\cite{GassotGM26} that, if the initial data $u_0\in H^1(\R)$ is sufficiently decaying at infinity,  the solution $u(t,\cdot )$ can be written as $t\to \infty$ as a finite superposition of soliton solutions with different velocities and a radiative remainder term. In this setting, the  number of solitons involved in this soliton resolution is shown to be finite thanks to the sufficient spatial decay of the initial data, in particular $xu_0\in L^2(\R)$. This includes the special case of multisoliton solutions, corresponding to initial data which are finite sums of functions $R_{p_j}$, and which was studied before in \cite{Matsuno79} and revisited in \cite{Sun2020}. In the latter case, the radiative remainder term is going to $0$  in any Sobolev space as $t\to \infty $.

  The goal of this paper is to investigate a similar long-time asymptotics for a family of initial data for which the number of solitons is infinite. In particular, these initial data will satisfy $u_0\in L^2(\R)$ but $x u_0\not\in L^2(\R)$.

\begin{definition}[Infinite-order multisoliton]\label{def:u0}
We fix a sequence $(p_j)_{j\geq 1}$ of points in the upper half-plane $\C_+$.  We say that $u_0$ is an infinite-order multisoliton in $L^2(\R)$ if $u_0$ is given by the formula
\begin{equation}\label{eq:u0}
u_0(x):=\sum_{j= 1}^{\infty}\frac{2\Im(p_j)}{|x+p_j|^2},
\quad x\in\R,
\end{equation}
and if the parameters $(p_j)_{j\geq 1}$ satisfy the condition
\begin{equation}\label{eq:norm}
\sum_{j=1}^{\infty}\sum_{k=1}^{\infty}\frac{\Im(p_j)}{|p_j-\overline{p_k}|^2}<\infty.
\end{equation}
\end{definition}

Our main result states that an infinite-order multisoliton in $L^2(\R)$ can be written as a superposition of solitons with different velocities in the long-time asymptotics, up to a small remainder term if the number of solitons is taken large enough.

\begin{remark}\label{rem:charact}
We can show that $u_0$ is an infinite-order multisoliton according to Definition~\ref{def:u0} if and only if $u_0$ belongs to the $L^2$-closure of finite multisolitons and is not a finite multisoliton. See Appendix~\ref{sec:appendix} for a proof.

As a consequence, one can show that the characterization given by Definition~\ref{def:u0} is propagated by the Benjamin-Ono flow. Indeed, let $u_0$ be an infinite-order multisoliton, and write $u_0$ as the $L^2$-limit of a family of multisolitons $(u_0^N)_{N\geq 1}$. Let $u_N$ and $u$ be the solutions to the Benjamin-Ono equation with initial data $u_N(0)=u_0^N$ and $u(0)=u_0$, respectively, and let $t\in\R$. By continuity of the Benjamin-Ono flow map in $L^2$ (see for instance~\cite{KillipLaurensVisan23-BO}), we deduce that $u_N(t)$ converges to $u(t)$ in $L^2(\R)$ as $N\to\infty$, and in particular $u(t)$ belongs to the $L^2$-closure of multisolitons.
\end{remark}

\begin{theorem}[Soliton resolution for the infinite-order multisoliton]\label{thm:main}
The initial data $u_0$ given by~\eqref{eq:u0} under the condition~\eqref{eq:norm} belong to $H^s(\R)$ for every $s\in \R$. Moreover, there is a sequence $(p_j^{\infty})_{j\geq 1}$ of complex numbers in $\C_+$ with increasing imaginary parts $0<\Im(p_1^\infty)<\Im(p_2^\infty)<\dots$, such that there holds
\begin{equation}\label{eq:resol-sol}
\lim_{N\to\infty}\limsup_{t\to \infty}\Big\|u(t,\cdot)-\sum_{j=1}^NR_{p_j^{\infty}}\left(\cdot-c_{p_k^{\infty}}t\right)\Big\|_{L^2}=0.
\end{equation}
Furthermore, for every $t\in \R$, the formula
\begin{equation}\label{eq:usol}
 u_{\mathrm{sol}}(t,x):=\sum_{j=1}^\infty R_{p_j^\infty}(x-c_{p_j^\infty}t),
 \quad x\in\R,
 \end{equation}
defines a continuous function $u_{\mathrm{sol}}(t,\cdot )$ on $\R$, tending to $0$ as $x\to \infty$, and belonging to all the homogeneous Sobolev spaces $\dot H^s(\R)$ for $s\geq 1/2$. Finally, we have, for every $s\geq 1/2$,
\begin{equation}\label{eq:resol-sol-bis}
\| u(t,\cdot )-u_{\mathrm{sol}}(t,\cdot )\|_{L^\infty}+\| u(t,\cdot )-u_{\mathrm{sol}}(t,\cdot )\|_{\dot H^s}\td_t,\infty 0.
\end{equation}
\end{theorem}
The above statement calls for several comments. Firstly,  the soliton resolution \eqref{eq:resol-sol} in $L^2$ has a different and slightly more complicated formulation compared to the soliton resolution 
\eqref{eq:resol-sol-bis} in $L^\infty$ and $\dot H^s$, because we do not know whether the function $u_{\mathrm{sol}}(t,\cdot )$ defined by \eqref{eq:usol} belongs to $L^2(\R)$. Secondly, Theorem \ref{thm:main} clearly defines a map $(p_j)_{j\geq 1}\mapsto (p_j^\infty)_{j\geq 1}$ from sequences of $\C_+$ satisfying \eqref{eq:norm} into a subset of sequences of $\C_+$ with strictly increasing imaginary parts. However, even if we know --- see section 5--- that
\[ \sum_{j=1}^\infty \frac{1}{\mathrm{Im}(p_j^\infty)}<\infty ,\]
 we do not know  the exact range of this map: in view of section 5, describing this range would be equivalent to solving an inverse spectral problem for 
 the Lax operators associated to potentials $u_0$ given by \eqref{eq:u0}, which we think is an interesting open problem. See also Remark \ref{rem:open} for a related inverse problem.
 
The  plan of the paper is as follows.
In Section~\ref{sec:notation}, we recall the Lax pair structure associated to the Benjamin--Ono equation, and an explicit formula for the Benjamin-Ono equation derived in~\cite{Gerard22}. In Section~\ref{sec:Lax}, we show that the eigenfunctions associated to the Lax operator $L_{u_0}$ belong to the domain of the  operator  $X^*$ occurring  in the explicit formula. In Section~\ref{sec:construction}, we show that the infinite-order multisoliton $u_0$ belongs to $H^s$ for every $s\geq 0$, and we derive an identity between the $L^2$ norm of $u_0$ and the eigenvalues of the Lax operator $L_{u_0}$. Finally, in Section~\ref{sec:resol-sol}, we prove Theorem~\ref{thm:main}.

\subsection*{Acknowledgements}
L. Gassot was supported by the France 2030 framework program, the Centre Henri Lebesgue ANR-11-LABX-0020-01, and the ANR project HEAD--ANR-24-CE40-3260.\\
P. G\'erard was partially supported by the French Agence Nationale de la Recherche under the ANR project ISAAC--ANR-23--CE40-0015-01.\\

\section{Explicit formula for the Benjamin-Ono equation}\label{sec:notation}

The proof of Theorem~\ref{thm:main} relies on an explicit formula derived by the second author in~\cite{Gerard22}, which we recall in this section. Let us denote by \[\langle f,g\rangle:=\int_{\R}f(x)\overline{g(x)}\dd x\] the inner product on $L^2(\R)$.

We denote by $L^2_+(\R)$ the Hardy space  of complex-valued functions $f\in L^2(\R,\C)$ whose Fourier transform is supported in the nonnegative half-line: $\mathrm{supp}\widehat{f}\subseteq [0,\infty)$.
In view of the inverse Fourier formula,
\[ f(z)=\frac{1}{2\pi}\int_0^\infty {\mathrm{e}}^{\ii z\xi}\widehat f(\xi)\dd\xi,\]
an element $f$ in $L^2_+(\R)$ extends as a holomorphic function on $\C_+$ satisfying
\[
\sup_{y>0}\int_{\R}|f(x+\ii y)|^2\dd x<\infty.
\]
Let $\Pi$ be the orthogonal projector from $L^2(\R)$ onto $L^2_+(\R)$. For $b\in L^{\infty}(\R)$, we denote by $T_b$ the Toeplitz operator on $L^2_+(\R)$ with symbol $b$, given by
\[
T_bf:=\Pi(bf),
\quad f\in L^2_+(\R).
\]
The Lax operator for~\eqref{BOinit} associated to $u_0$ is defined as 
\[
L_{u_0}:=-\ii\frac{\dd }{\dd x}-T_{u_0}.
\]
It is a self--adjoint operator on $L^2_+(\R)$, with the domain $H^1_+(\R):=L^2_+(\R)\cap H^1(\R)$. 
Next, we introduce the operator $X^*$ defined as the adjoint in $L^2_+(\R)$ of the multiplication by~$x$. Alternatively, one can define $X^*$ through the formula
\[
\widehat{X^*f}(\xi):=\ii\frac{\dd }{\dd \xi}\widehat{f}(\xi),
\quad \xi>0,
\]
and the domain
\[
\mathcal{D}(X^*)=\{f\in L^2(\R) \;\colon\; \widehat{f}|_{(0,\infty)}\in H^1(0,\infty)\},
\]
see for instance~\cite{Gerard26}.

Finally, for $f\in\mathcal{D}(X^*)$, the Fourier transform $\widehat{f}(\xi)$ admits a limit as $\xi\to 0^+$ that we denote by
\[
I_+(f):=\widehat{f}(0^+).
\]

The explicit formula for the Benjamin-Ono equation reads as follows.

\begin{theorem}[Explicit formula, see~\cite{Gerard22}]
Let $u_0\in H^1(\R)$,  then the solution $u\in\mathcal{C}(\R,H^1(\R))$ of~\eqref{BOinit} is given by $u(t,x)=2\Re(\Pi u(t,x))$, where the holomorphic extension of $\Pi u$ to the complex upper half-plane $\C_+$ satisfies:  for every $t\in\R$ and $z\in\C_+$,
\begin{equation}\label{eq:explicit}
\Pi u(t,z)=\frac{1}{2\ii\pi}I_+((X^*-2tL_{u_0}-z)^{-1}\Pi u_0).
\end{equation}
\end{theorem}

\section{Eigenfunctions of the Lax operator}\label{sec:Lax}

In this section, we prove a general property of the Lax eigenfunctions associated to a potential in $L^2(\R)$, which may be of independent interest.

\begin{proposition}[Lax eigenfunctions belong to the domain of $X^*$]\label{prop:phi-domain}
Let $u_0\in L^2(\R)$, and let $\varphi$ be an eigenfunction of $L_{u_0}$ associated to a negative eigenvalue $\lambda<0$. Then $\varphi\in\mathcal{D}(X^*)$.
Moreover $X^*\varphi \in \mathcal {D}(L_{u_0})$ and 
\begin{equation}\label{eq:Xstarphi}
(L_{u_0}-\lambda)(X^*\varphi )=-\ii\varphi +\frac{\ii}{2\pi |\lambda|}\langle \varphi ,\Pi u_0\rangle \Pi u_0.
\end{equation}
\end{proposition}

\begin{proof}
Taking the Fourier transform of the equality $L_{u_0}\varphi=\lambda\varphi$, we find
\begin{equation}\label{eq:phi-fourier}
\forall \xi>0,\quad \widehat{\varphi}(\xi)=\frac{\widehat{u_0\varphi}(\xi)}{(\xi+|\lambda|)}.
\end{equation}
Since $u_0\varphi \in L^1$, it is straightforward to check that the right-hand side of \eqref{eq:phi-fourier} belongs to $L^2(0,\infty)$ and is continuous on $[0,\infty)$, so that
\begin{equation}\label{I+phi}
I_+(\varphi )=\widehat \varphi (0^+)=\frac{\widehat{u_0\varphi}(0)}{|\lambda|}=\frac{\langle \varphi, \Pi u_0\rangle }{|\lambda|}.
\end{equation}

In order to show that the right-hand side of \eqref{eq:phi-fourier} belongs to $H^1(0,\infty)$, we will show that the growth rate $\tau_\eta\widehat{\varphi}$, where $\tau_\eta$ is defined for $\eta>0$ as
\[
\tau_\eta \psi:=\frac{\psi(\xi+\eta)-\psi(\xi)}{\eta},
\quad \psi\in L^2(0,\infty),
\]
admits a limit as $\eta\to 0$ in $L^2(0,\infty)$.

 First notice that the right hand side of \eqref{eq:phi-fourier} also reads $A_{u_0}(\widehat \varphi) (\xi )$, where we have set
 \[ A_v \psi (\xi )=\frac{1}{2\pi (\xi +|\lambda |)}\int_0^\infty \widehat v(\xi -\zeta)\psi (\zeta)\dd\zeta ,\]
 for $v\in L^2(\R), \psi \in L^2(0,\infty )$.
Similarly to what we already observed, $A_v $ is a bounded operator from $L^2(0,\infty )$ into itself, with a norm bounded by $C\Vert v\Vert_{L^2}$, for some constant $C$ depending only on $\lambda $.

Let us fix some parameter $R>0$ and introduce the truncated function  $u_{0,R}$ defined as
\begin{equation}\label{eq:u0R}
u_{0,R}(x)=\mathbf{1}_{|x|<R}u_0(x),
\quad
\forall x\in\R.
\end{equation}
We decompose the eigenvalue equation~\eqref{eq:phi-fourier} as
\begin{equation}\label{eq:phi-fourierbis}
\forall \xi>0, \quad \widehat{\varphi}(\xi)-A_{u_0-u_{0,R}}(\widehat \varphi) (\xi)=A_{u_{0,R}}(\widehat \varphi )(\xi).
\end{equation}
Since the norm of the operator $A_{u_0-u_{0,R}}$ tends to $0$ as $R\to \infty$, we can select $R>0$ so that $\mathrm{Id}-A_{u_0-u_{0,R}}$ is invertible as an operator of $L^2(0,\infty )$. Next we identify the commutator of $\tau_\eta $ with $A_v$. A simple calculation leads to
\begin{equation}\label{eq:tauA}
[\tau_\eta, A_v]\psi (\xi)=-\frac{A_v\psi (\xi +\eta)}{\xi +|\lambda|}+\frac{1}{2\pi(\xi +|\lambda|)}\int_0^\eta \widehat v(\xi +\eta -\zeta)\psi (\zeta)\, \frac{\dd\zeta}{\eta}.
\end{equation}
If $\psi $ is continuous on $[0,\infty )$, we conclude that $[\tau_\eta ,A_v]\psi $ has a limit $\Lambda _v\psi $ in $L^2(0,\infty )$ as $\eta \to 0^+$, given by
\begin{equation}\label{eq:limit}
 \Lambda_v\psi (\xi )=-\frac{A_v\psi (\xi)}{\xi +|\lambda|}+\frac{\widehat v(\xi )\psi (0^+)}{2\pi(\xi +|\lambda|)}. 
 \end{equation}
 Coming back to \eqref{eq:phi-fourierbis}, we have
 \[
\tau_\eta\widehat{\varphi}-A_{u_0-u_{0,R}}(\tau_\eta \widehat \varphi) =\tau_\eta A_{u_{0,R}}(\widehat \varphi )+[\tau_\eta, A_{u_0-u_{0,R}}](\widehat \varphi ),\]
 and, since $xu_{0,R}\in L^2$ and $\widehat \varphi $  is continuous on $[0,\infty )$, the right hand side has a limit in $L^2(0,\infty )$ as $\eta \to 0^+$, given by
 \[-\frac{1}{\xi +|\lambda|}A_{u_0,R}(\widehat \varphi)-\ii A_{xu_{0,R}}(\widehat \varphi )+ \Lambda_{u_0-u_{0,R}}(\widehat \varphi ).\]
 Consequently,  the growth rate satisfies, as $\eta \to 0^+$, 
\[ \tau_\eta \widehat \varphi \longrightarrow (\mathrm{Id}-A_{u_0-u_{0,R}})^{-1}\left (-\frac{1}{\xi +|\lambda|}A_{u_0,R}(\widehat \varphi)-\ii A_{xu_{0,R}}(\widehat \varphi )+ \Lambda_{u_0-u_{0,R}}(\widehat \varphi )\right ).\]
This completes the proof of the first statement. For the second statement, let us apply $\tau_\eta $ directly to \eqref{eq:phi-fourier},
\[ \tau_\eta \widehat \varphi =A_{u_0}\tau_\eta \widehat \varphi +[\tau_\eta ,A_{u_0}]\widehat \varphi .\]
Taking the limit as $\eta \to 0^+$, we obtain, in view of \eqref{eq:limit} and~\eqref{eq:phi-fourier}, 
\[ \quad  \widehat \varphi '(\xi )=A_{u_0}( \widehat \varphi ')(\xi )+\Lambda_{u_0}(\widehat \varphi )(\xi)=A_{u_0}( \widehat \varphi ')(\xi )-(\xi +|\lambda|)^{-1}\widehat \varphi (\xi)+(2\pi(\xi +|\lambda|))^{-1}\widehat u_0(\xi )I_+(\varphi ).\]
This equation exactly means
\[ (L_{u_0}-\lambda)(X^*\varphi )=-\ii \varphi +\frac{\ii}{2\pi}I_+(\varphi )\Pi u_0,\]
which, in view of \eqref{I+phi}, is  \eqref{eq:Xstarphi}.
\end{proof}
Using formula \eqref{eq:Xstarphi}, one easily recovers Wu's identity ~\cite[Lemma 2.5]{Wu16a}, see also ~\cite{BadreddineKV25}. 
\begin{corollary}\label{cor:Wu}
Let $u_0\in L^2(\R)$, and let $\varphi$ be an eigenfunction of $L_{u_0}$ associated to a negative eigenvalue $\lambda<0$. Then
\[ |\langle \Pi u_0,\varphi \rangle |^2=2\pi |\lambda |\Vert \varphi \Vert ^2.\]
\end{corollary}
\begin{proof}
Take the inner product with $\varphi $ of both sides of \eqref{eq:Xstarphi}.
\end{proof}

\section{Construction of the infinite-order multisoliton}\label{sec:construction}

In this section, we justify that the initial data defined in~\eqref{eq:u0} belongs to $H^s$ for $s\geq 0$, and that it admits an infinite number of negative Lax eigenvalues.

\begin{lemma}[Finite $H^s$ norm for $s\geq 0$]\label{lem:L2}
The function $u_0$ given by~\eqref{eq:u0} belongs to $L^2(\R)$ if and only if the sum~\eqref{eq:norm} is finite.
In this case,  $u_0$ belongs to  $H^s(\R)$ for every $s\geq 0$.
\end{lemma}

\begin{proof}
We define the partial sums
\begin{equation}\label{eq:u0N}
u_0^N(x):=\sum_{j=1}^N \frac{2\Im(p_j)}{|x+p_j|^2}.
\end{equation}
It is enough to show that the sequence $(\Pi u_0^N)_{N\geq 1}$ is a Cauchy sequence in $L^2(\R)$. Note that for every $x\in\R$, we have
\[
\Pi u_0^N(x)
	=\sum_{j=1}^N\frac{\ii}{x+p_j}.
\]
Let $1\leq M<N$. Applying the residue formula, we find
\begin{equation*}
\|\Pi u_0^N-\Pi u_0^M\|_{L^2}^2
	=\sum_{j,k=M+1}^{N}\int_{\R}\frac{\dd x}{(x+p_j)(x+\overline{p_k})}
	=\sum_{j,k=M+1}^{N} \frac{2\ii \pi}{p_j-\overline{p_k}}.
\end{equation*}

Comparing the indices $j$ and $k$, we conclude
\begin{align}\label{eq:Cauchy}
\|\Pi u_0^N-\Pi u_0^M\|_{L^2}^2
	&=\sum_{j=M+1}^N \frac{\pi}{\Im(p_j)}
	+\sum_{M<j<k\leq N}\frac{4\pi(\Im(p_j)+\Im(p_k))}{|p_j-\overline{p_k}|^2}\\ \nonumber
	&=\sum_{j,k=M+1}^N \frac{4\pi\Im(p_j)}{|p_j-\overline{p_k}|^2}.
\end{align}
This implies that $(\Pi u_0^N)_N$ is a Cauchy sequence in $L^2(\R)$ if and only if the positive series~\eqref{eq:norm} is convergent. In particular, note that thanks to~\eqref{eq:Cauchy}, we have
\begin{equation}\label{eq:finite-sum}
\sum_{j=1}^\infty \frac{1}{\Im (p_j)}<+\infty.
\end{equation}
This implies that the sequence $(\Im(p_j))_{j\geq 1}$ should go to infinity, hence it is bounded from below by some positive number $c_0>0$.

We finally fix $s\in (0,\infty)$ and we calculate 
\[ \Vert R_p\Vert_{\dot H^s}^2=\frac 12 \Vert \Pi R_p\Vert_{\dot H^s}^2=\frac{1}{2\pi}\int_0^\infty \xi ^{2s}|\widehat{\Pi R_p}(\xi)|^2\dd\xi ,\]
where we recall that $R_p$ is given by \eqref{eq:Rp}. 
Since 
\[ \widehat{\Pi R_p}(\xi )=2\pi {\mathrm{e}}^{\ii p\xi},
\quad \xi >0,
\]
we infer
\begin{equation}\label{eq:RpdotHs}
\Vert R_p\Vert_{\dot H^s}^2 =\pi \int_0^\infty \xi ^{2s}{\mathrm{e}}^{-2\Im(p)\xi}\dd\xi =\frac{\pi \Gamma (2s+1)}{(2\Im (p))^{2s+1}}.
\end{equation}
Using~\eqref{eq:finite-sum}, we deduce that the series $\| R_{p_j}\|_{\dot H^s}$ is convergent for every $s\geq 1/2$, and therefore $u_0\in \dot H^s(\R)$ for $s\geq 1/2$. Since $u_0\in L^2(\R)$, this completes the proof.
\end{proof}

\begin{remark}
Note that $u_0$ does not belong to $L^1(\R)$, and in particular $xu_0\not\in L^2(\R)$. Indeed, we have
\[
\|u_0^N\|_{L^1}
	=\sum_{j=1}^N \int_{\R}\frac{2\Im(p_j)}{|x+p_j|^2}\dd x=2N\pi,
\] 
hence the sequence $(u_0^N)_{N\geq 1}$ does not converge in $L^1(\R)$.
\end{remark}

\begin{proposition}[Infinite number of negative Lax eigenvalues]
Consider the initial data $u_0$ given by~\eqref{eq:u0}. The Lax operator $L_{u_0}$ admits an infinite number of negative eigenvalues $(\lambda_k)_{k\geq 1}$ with $\lambda_1<\lambda_2<\dots<0$. Moreover, there holds the distorted Plancherel formula
\begin{equation}\label{eq:Plancherel}
\sum_{k= 1}^{\infty}2\pi|\lambda_k|=\|\Pi u_0\|_{L^2}^2.
\end{equation}
\end{proposition}

\begin{proof}
We know ---see e.g. \cite{Wu16a}--- that the essential spectrum of $L_{u_0}$ is the half--line $[0,\infty )$. Applying the max-min formula for the eigenvalues of $L_{u_0}$, we find that whenever the following maximum is negative, the operator $L_{u_0}$ admits at least $N$ negative eigenvalues, and
\[
\lambda_N=\max_{\mathrm{dim} F=N-1} \min\{\langle L_{u_0}h|h\rangle \;\colon\; h\in H^1_+(\R)\cap F^\perp, \; \|h\|_{L^2}=1\}.
\]
Now, we fix $N\geq 1$ and we show that $u_0$ admits at least $N$ eigenvalues by comparison with the truncated initial data $u_0^N$ given by~\eqref{eq:u0N}. Given that $u_0^N$ is a $N$-soliton potential, it admits exactly $N$ negative eigenvalues $\lambda_1^N<\dots<\lambda_N^N<0$, see e.g. \cite{Sun2020}.\\ Moreover, since $u_0(x)\geq u_0^N(x)$ for every $x\in\R$, we know that $\langle L_{u_0}h|h\rangle \leq \langle L_{u_0^N}h|h\rangle$ for every $h\in H^1_+(\R)$. Hence the max-min formula implies that $u_0$ admits at least $N$ negative eigenvalues, satisfying $\lambda_k\leq \lambda_k^N$ for $1\leq k\leq N$. 

We now show equality~\eqref{eq:Plancherel}. Since $u_0^N$ is a multi--soliton of order $N$,  we know --- see e.g. \cite{Sun2020}--- that  $\Pi u_0^N$ is a linear combination of the $N$ normalized eigenfunctions $\varphi_1^N,\dots, \varphi_N^N$ of $L_{u_0^N}$ associated to $\lambda_1^N,\dots, \lambda_N^N$, and Wu's identity, recalled in Corollary \ref{cor:Wu}, implies that
\[
\|\Pi u_0^N\|_{L^2}^2=\sum_{k=1}^N |\langle \Pi u_0^N, \varphi_j^N)|^2=\sum_{k=1}^N2\pi|\lambda_k^N|,
\]
so that
\[
\|\Pi u_0^N\|_{L^2}^2
	\leq \sum_{k=1}^{\infty}2\pi|\lambda_k|.
\]
Passing to the limit $N\to\infty$, we deduce the inequality 
\(
\|\Pi u_0\|_{L^2}^2
	\leq \sum_{k=1}^{\infty}2\pi|\lambda_k|.
\)

In order to show the converse inequality, we apply Wu's identity to $u_0$ itself. Let $\varphi_k$ be an eigenfunction of $L_{u_0}$ associated to the negative eigenvalue $\lambda_k$ and satisfying $\|\varphi_k\|_{L^2}=1$.  Since the family $(\varphi_k)_{k\geq 1}$ is orthonormal, we deduce that 
\[
\|\Pi u_0\|_{L^2}^2
	\geq \sum_{k=1}^{\infty} |\langle \Pi u_0,\varphi_k\rangle|^2=\sum_{k=1}^{\infty}2\pi|\lambda_k|.
\]
We conclude that~\eqref{eq:Plancherel} holds.
\end{proof}

\begin{remark}\label{rem:open}
In view of the distorted Plancherel formula~\cite[Theorem 26]{BlackstoneGGM24a} proven for initial data $u_0$ with $xu_0\in L^2(\R)$, it is natural to ask whether all initial data $u_0\in L^2(\R)$ such that the eigenvalues  of the Lax operator $L_{u_0}$ satisfy~\eqref{eq:Plancherel} are of the form~\eqref{eq:u0}. This seems to be an open problem.
\end{remark}

\section{Soliton resolution for the infinite-order multisoliton}\label{sec:resol-sol}

In this section, we prove Theorem~\ref{thm:main}.  We adapt the strategy implemented in~\cite{GassotGM26}.

For $f\in L^2_+(\R)$ and $t\in\R$, we introduce an operator $\Omega_t$ on $L^2_+(\R)$  as 
\[
\Omega_tf(z)=\frac{1}{2\ii\pi}I_+((X^*-2tL_{u_0}-z)^{-1}f),
\quad \forall z\in\C_+,
\]
so that the explicit formula \eqref{eq:explicit} reads $\Pi u(t,z)=\Omega_t(\Pi u_0)(z)$.
We first note that~\cite[Lemma 4.1]{GassotGM26} still holds.
\begin{lemma}[See~\cite{GassotGM26}]\label{lem:Omega-t}
For $f\in L^2_+(\R)$, then $\Omega_t f\in L^2_+(\R)$ and we have
\[
|\Omega_tf(z)|
	\leq \frac{\|f\|_{L^2}}{2\sqrt{\pi \Im(z)}},
	\quad \forall z\in\C_+.
\]
\end{lemma}

In what follows, we  refine the proof of~\cite[Lemma 4.2]{GassotGM26} in order to remove the assumption $xu_0\in L^2(\R)$.

\begin{lemma}\label{lem:perp}
Let $j\geq 1$ and $f\in L^2_+(\R)$ such that $f\perp \varphi_j$. Then
\[
\Omega_tf(z-2t\lambda_j)\longrightarroww{t\to\infty}{} 0,
\quad
\forall z\in\C_+.
\]
\end{lemma}

\begin{proof}
By density of $\mathcal{D}(X^*)\cap\varphi_j^\perp$ in $\varphi_j^\perp$ and the uniform bound on $\Omega_t f$ given in Lemma~\ref{lem:Omega-t}, it is enough to prove the result when $f\in\mathcal{D}(X^*)\cap \varphi_j^\perp$. In this case, by functional calculus for the selfadjoint operator $L_{u_0}$, there exists $h\in L^2_+(\R)$ such that
\[
(L_{u_0}-\lambda_j)h=f.
\]
Let us show that $h\in\mathcal{D}(X^*)$. By applying the Fourier transform to the above equation, we find
\begin{equation}\label{eq:h-fourier} 
\forall \xi>0,\quad 
\widehat{h}(\xi)=\frac{\widehat{f}(\xi)}{\xi-\lambda_j}+\frac{\widehat{u_0h}(\xi)}{2\pi(\xi-\lambda_j)}.
\end{equation}
At this stage, we just adapt the proof of Proposition \ref{prop:phi-domain}. The only difference of \eqref{eq:h-fourier} with respect to \eqref{eq:phi-fourier} is the additional term $(\xi -\lambda_j)^{-1}\widehat f(\xi )$, which obviously belongs to $H^1(0,\infty )$ since $f\in \mathcal {D}(X^*)$. Hence the same proof leads to $\widehat h\in H^1(0,\infty )$, namely $h\in \mathcal{D}(X^*).$

The conclusion now follows from the proof of~\cite[Lemma 4.2]{GassotGM26}: assuming $f\in\mathcal{D}(X^*)\cap\varphi_j^\perp$ so that $h\in\mathcal{D}(X^*)$, we write for $z\in\C_+$
\[
\Omega_tf(z-2t\lambda_j)
	=-\frac{I_+(h)}{4\ii\pi t}+\frac{\Omega_t(X^*h-zh)(z-2t\lambda_j)}{2t},
\]
which goes to zero as $t\to+\infty$ thanks to Lemma~\ref{lem:Omega-t}.
\end{proof}

Next, we adapt the proof of~\cite[Lemma 3.2]{GassotGM26} and infer that   $\Pi u(t,z-2t\lambda_j)$ admits a pointwise limit as $t\to+\infty$. 
\begin{lemma}\label{lem:pointwise}
Let $j\geq 1$ and $z\in\C_+$. Then as $t\to\infty$, we have
\[
\Pi u(t,z-2t\lambda_j)\longrightarroww{t\to\infty}{} \Pi R_{p_j^\infty}(z),
\quad
p_j^{\infty}:=-\langle X^*\varphi_j,\varphi_j\rangle.
\]
\end{lemma}
As we  have already observed, applying the operator $\Pi$ to~\eqref{eq:Rp} leads to $\Pi R_{p_j^{\infty}}(z)=\frac{\ii}{z+p_j^{\infty}}$. Moreover, from \eqref{I+phi}, we know that 
\begin{equation}\label{eq:impinfty}
 \Im( p_j^\infty) =-\Im \langle X^*\varphi_j,\varphi_j\rangle =-\frac{1}{2\pi}\Re \langle \frac{\dd}{\dd\xi}\widehat \varphi_j,\widehat\varphi_j\rangle =\frac{|I_+(\varphi_j)|^2} {4\pi}=\frac{1}{2|\lambda_j|}.
\end{equation}
Let us now recall the ideas of proof of this Lemma.
\begin{proof}
We decompose $\Pi u_0=\langle \Pi u_0,\varphi_j\rangle \varphi_j+f_j$, where $f_j\in\varphi_j^\perp$. Applying Lemma~\ref{lem:perp}, we know that $\Omega_tf_j(z-2t\lambda_j)\to 0$ as $t\to\infty$.
Next, set
\[
\psi_j:=\frac{\varphi_j}{\langle X^*\varphi_j,\varphi_j\rangle-z},
\quad
g_j:=\frac{X^*\varphi_j-\langle X^*\varphi_j,\varphi_j\rangle\varphi_j}{\langle X^*\varphi_j,\varphi_j\rangle-z},
\]
and observe that
\[
(X^*-2t(L_{u_0}-\lambda_j)-z)\psi_j=\varphi_j+g_j,
\]
so that
\[
\Omega_t(\varphi_j)(z-2t\lambda_j)=\frac{I_+(\psi_j)}{2\ii\pi}-\Omega_t(g_j)(z-2t\lambda_j).
\]
Since $g_j\in\varphi_j^{\perp}$, Lemma~\ref{lem:perp} implies that $\Omega_t(g_j)(z-2t\lambda_j)$ as $t\to+\infty$. Moreover,
\[
I_+(\psi_j)=\frac{I_+(\varphi_j)}{\langle X^*\varphi_j,\varphi_j\rangle -z}.
\]
Thanks to Corollary \ref{cor:Wu} and \eqref{I+phi}, we have $\langle \Pi u_0,\varphi_j\rangle I_+(\varphi_j)=2\pi$. Combining all the identities together, we conclude that
\[
\Omega_t(\Pi u_0)(z-2t\lambda_j)\longrightarroww{t\to\infty}{}\frac{\ii}{z-\langle X^*\varphi_j,\varphi_j\rangle}.
\qedhere
\]
\end{proof}

\begin{proof}[Proof of Theorem~\ref{thm:main}]
Let us fix $N\geq 0$, and define the remainder term
\[
 r_N(t,x):=u(t,x)-\sum_{j=1}^NR_{p_j^\infty}(x-c_{p_j^\infty}t),
 \quad
c_{p_j^\infty}=\frac{1}{\Im p_j^\infty}=2|\lambda_j|.
\]
We expand
\begin{multline*}
\|\Pi r_N(t,\cdot)\|_{L^2}^2
	=\|\Pi u(t,\cdot)\|_{L^2}^2+\sum_{j=1}^N \|\Pi R_{p_j^\infty}\|_{L^2}^2
	-2\sum_{j=1}^N\Re\langle \Pi u(t,\cdot),\Pi R_{p_j^\infty}(\cdot+2t\lambda_j)\rangle
	\\
	-2\sum_{\substack{j,k=1\\ j\neq k}}^N\Re\langle \Pi R_{p_j^\infty}(\cdot+2t\lambda_j),\Pi R_{p_k^\infty}(\cdot+2t\lambda_k)\rangle.
\end{multline*}
Given the weak convergence $\Pi u(t,\cdot-2t\lambda_j)\rightharpoonup \Pi R_{p_j^\infty}$ provided by Lemma~\ref{lem:pointwise}, we find that for every $j\geq 1$,
\[
\langle \Pi u(t,\cdot),\Pi R_{p_j^\infty}(\cdot+2t\lambda_j)\rangle
	=\langle \Pi u(t,\cdot-2t\lambda_j),\Pi R_{p_j^\infty}\rangle
	\longrightarroww{t\to\infty}{} \langle \Pi R_{p_j^\infty},\Pi R_{p_j^\infty}\rangle.
\]
Moreover, when $j\neq k$, we have $\lambda_j\neq \lambda_k$, and in particular
\[
\langle \Pi R_{p_j^\infty}(\cdot+2t\lambda_j),\Pi R_{p_k^\infty}(\cdot+2t\lambda_k)\rangle
	\longrightarroww{t\to\infty}{} 0.
\]
Combining these observation with the conservation of the $L^2$ norm under the flow of~\eqref{BOinit},  this leads to the following identity as $t\to+\infty$:
\begin{equation*}
\|\Pi r_N(t,\cdot)\|_{L^2}^2
	=\|\Pi u_0\|_{L^2}^2
	-\sum_{j=1}^N \|\Pi R_{p_j^\infty}\|_{L^2}^2
	+o_{t\to\infty}(1).
\end{equation*}
Thanks to~\eqref{eq:Plancherel} and the identity $\|\Pi R_{p_j^\infty}\|_{L^2}^2=2\pi|\lambda_j|$, we find that
\begin{equation*}
\|\Pi r_N(t,\cdot)\|_{L^2}^2
	=\sum_{j=N+1}^{\infty} 2\pi|\lambda_j|
	+o_{t\to\infty}(1),
\end{equation*}
where the first term in the right-hand side is the remainder term of a convergent series in $N$. Letting $t\to\infty$ then $N\to\infty$, we find~\eqref{eq:resol-sol}.

Let us prove \eqref{eq:resol-sol-bis}. First, in view of \eqref{eq:impinfty} and of Proposition \eqref{eq:Plancherel}, we have
\[ \sum_{j=1}^\infty \frac{1}{\Im (p_j^\infty)} <\infty .\]
Since 
\[\| R_p\|_{L^\infty}=\frac{2}{\Im (p)},\]
we infer that the series \eqref{eq:usol} defining $u_{\mathrm{sol}}(t,x)$ is uniformly convergent, so that $u_{\mathrm{sol}}(t,\cdot )$ is continuous on $\R$ and tends to $0$ at infinity. Furthermore, because of \eqref{eq:RpdotHs}, the series \eqref{eq:usol} also converges normally in $\dot H^s(\R)$ for every $s\geq 1/2$. Hence $u_{\mathrm{sol}}(t,\cdot )\in \dot H^s$ for every $s\geq 1/2$. 
On the other hand, due to the Benjamin--Ono conservation laws --- see e.g. \cite{Gerard26} ---, we know that $u(t,\cdot )$ is bounded in all the Sobolev spaces $H^k (\R)$ for $k\geq 0$. 
Given $N\geq 1$, set
\[ u_{\mathrm{sol}}^N(t,x):=\sum_{j=1}^N R_{p_j^\infty}(x-c_{p_j^\infty}t).\]
Recall that thanks to~\eqref{eq:RpdotHs}, we have $\|R_{p_j^\infty}\|_{\dot{H}^s}^2=C_s\Im (p_j^\infty)^{-s-1/2}$ for some constant $C_s>0$, and that according to~\eqref{eq:impinfty}, we have $\Im(p_j^\infty)=(2|\lambda_j|)^{-1}$. Given $s\geq 1/2$, we get
 \begin{align*}
 \Vert u(t,\cdot)-u_{\mathrm{sol}}(t,\cdot)\Vert _{\dot H^s}&\leq  \Vert u(t,\cdots)-u_{\mathrm{sol}}^N(t,\cdot)\Vert _{\dot H^s}+ \Vert u_{\mathrm{sol}}^N(t,\cdot)-u_{\mathrm{sol}}(t,\cdot)\Vert _{\dot H^s}\\
 &\leq  \Vert u(t,\cdot)-u^N_{\mathrm{sol}}(t,\cdot)\Vert _{L^2}^{1/2} \Vert u(t,\cdots)-u^N_{\mathrm{sol}}(t,\cdot)\Vert _{\dot H^{2s}}^{1/2}+C_s\sum_{j=N+1}^\infty |\lambda_j|^{s+1/2}\\
 &\leq \tilde C_s (\Vert u(t,\cdot)-u^N_{\mathrm{sol}}(t,\cdot)\Vert _{L^2} ^{1/2}+\sum_{j=N+1}^\infty |\lambda_j|^{s+1/2}).
 \end{align*}
 Similarly, since $\Vert v\Vert_{L^\infty}^2\leq \Vert v\Vert_{L^2} \Vert v\Vert_{\dot H^1}$, we have
\[
  \Vert u(t,\cdot)-u_{\mathrm{sol}}(t,\cdot)\Vert _{L^\infty}\leq C (\Vert u(t,\cdot)-u^N_{\mathrm{sol}}(t,\cdot)\Vert _{L^2} ^{1/2}+\sum_{j=N+1}^\infty |\lambda_j|).
\]
Passing to the upper limit as $t\to \infty$, we obtain, for every $N\geq 1$,
 \[ \limsup_{t\to \infty} (  \Vert u(t,\cdots)-u_{\mathrm{sol}}(t,\cdot)\Vert _{L^\infty}+\Vert u(t,\cdot)-u_{\mathrm{sol}}(t,\cdot)\Vert _{\dot H^s})\lesssim  \delta_N^{1/2}+\sum_{j=N+1}^\infty |\lambda_j|,\]
 where $\delta_N:=\limsup_{t\to \infty}\Vert u(t,\cdot)-u^N_{\mathrm{sol}}(t,\cdot)\Vert _{L^2}$. 
Hence \eqref{eq:resol-sol-bis}  follows by letting $N\to \infty$. 
\end{proof}

\begin{remark}
In view of the explicit formula for the Benjamin-Ono hierarchy given by~\cite{GerardHe26hierarchy} (see also~\cite{KillipLaurensVisan23-BO}), it seems that the proof of Theorem~\ref{thm:main} adapts to the equations of the Benjamin-Ono hierarchy associated to the higher-order Hamiltonians $E_k$. The proof also adapts to equations obtained from finite linear combinations of these conservation laws, of the form 
\[\partial_t u=\partial_x\nabla\left(\sum_{k=1}^K \alpha_k E_k(u)\right).
\]
The main modification is the fact that the speed of the traveling wave $c_{p_k^{\infty}}$ from Theorem~\ref{thm:main} should be replaced by the corresponding speed $\widetilde{c}_{p_k^{\infty}}$ in the hierarchy according to~\cite{GerardHe26hierarchy}:
\[
\widetilde{c}_{p_k^{\infty}}:=\sum_{k=1}^K \alpha_k c_{k,p_{k}^\infty},
\quad
c_{k,p}:=(-1)^{k+1}\frac{k+1}{(2\Im(p))^n}.
\]
\end{remark}

\appendix
\section{A characterization of infinite-order multisolitons}\label{sec:appendix}

In this section, we prove the main statement in Remark~\ref{rem:charact}, which we state as follows.

\begin{proposition}
Denote by $\mathscr M$ the set of finite multisolitons, namely functions of the form
\[ u_0(x)=\sum_{j=1}^N \frac{2\Im{p_j}}{|x+p_j|^2},\]
where $N$ is a positive integer and $p_1,\dots,p_N$ are elements of $\C_+$. The closure of $\mathscr M$ in $L^2(\R)$ consists of the union of $\mathscr M$ and of the set of  infinite--order multisolitons according to Definition~\ref{def:u0}.
\end{proposition}
\begin{proof}
We first note that thanks to Section~\ref{sec:construction}, any $u_0$ satisfying Definition~\ref{def:u0}, is the $L^2$-limit of its partial sums $u_0^N$ defined in~\eqref{eq:u0N}, which are finite multisolitons.

Conversely, let $u_0\in L^2(\R)$ be the $L^2$-limit of a sequence $(u_0^N)_{N\geq 1}$ of finite multisolitons. Writing
\[
u_0^N(x)=\sum_{j=1}^N \frac{2\Im(p_j^N)}{|x+p_j^N|^2},
\]
one can see that
\begin{equation}\label{eq:u0N-theta}
u_0^N(x)=\frac{-\ii \Theta_N'(x)}{\Theta_N(x)}
\end{equation}
where $\Theta_N(x)$ is the finite Blaschke product given by
\[
\Theta_N(x)=\prod_{j=1}^N\frac{x+\overline{p_j^N}}{x+p_j^N}.
\]
Note that $|\Theta_N(x)|=1$ for every $x\in\R$, and that $\Theta_N$ extends as a holomorphic function to the upper half-plane  $\C_+$ with $|\Theta_N(z)|\leq 1$ for every $z\in\C_+$.

We first show that $(\Theta_N)_N$ is locally uniformly convergent to a limiting function $\Theta$.
Integrating the equation~\eqref{eq:u0N-theta} satisfied by $\Theta_N$, we can write for every $x\in\R$
\[
\Theta_N(x)=\Theta_N(0)\exp\left(\ii\int_0^x u_0^N(y)\dd y\right).
\]
Since $|\Theta_N(0)|=1$, up to some subsequence, one can assume that $(\Theta_N(0))_N$ is convergent to some limit $\Theta(0)$ with $|\Theta(0)|=1$.
Since the sequence $(u_0^N)_N$ converges to $u_0$ in $L^2(\R)$, we also know that  the sequence $(\Theta_N)_N$ is locally uniformly convergent on $\R$ to some limit $\Theta$ as $N\to\infty$. In particular, $\Theta$ is continuous on $\R$ and $|\Theta(x)|=1$ for every $x\in\R$.

Next, we show that $\Theta$ is an inner function, in the sense that it admits a holomorphic extension to $\C_+$ with $|\Theta(z)|\leq 1$ for every $z\in \C_+$, and that $|\Theta(x)|=1$ for every $x\in\R$. We observe that
\[
\|u_0^N\Theta_N-u_0\Theta\|_{L^2}\leq \|u_0^N-u_0\|_{L^2}\|\Theta_N\|_{L^\infty}+\| u_0(\Theta_N-\Theta)\|_{L^2},
\]
and the dominated convergence theorem implies that $\Theta_N'=\ii u_0^N\Theta_N$ converges to $\ii u_0\Theta$ in $L^2(\R)$. Since $\Theta_N'$ also tends to $\Theta'$ in the sense of distributions, we conclude that $\Theta'=\ii u_0\Theta\in L^2(\R)$.
Moreover, we have seen that $|\Theta_N(z)|\leq 1$ for $z\in\C_+$. The Montel theorem implies that up to some subsequence, $\Theta_N$ has a pointwise limit that we still denote $\Theta$, and that the convergence is uniform on compact subsets of $\C_+$. Moreover, $\Theta$ is holomorphic on $\C_+$ and $|\Theta(z)|\leq 1$ for every $z\in\C_+$.
Given that $\Theta_N'$ is bounded in $L^2_+(\R)$ we get from the inverse Fourier formula and the Cauchy-Schwarz inequality that  there is $C>0$ such that for every $x\in\R$ and $y>0$, 
\[
|\Theta_N'(x+\ii y)|\leq C\left(\int_0^\infty \ee^{-2y\xi}\dd\xi\right)^{1/2}\|\widehat{\Theta}_N'\|_{L^2(\R_+)}
\leq \frac{C'}{\sqrt{y}},
\]
so that  for every $\delta>0$ and $x\in\R$,
\[
|\Theta_N(x+\ii\delta)-\Theta_N(x)|	\leq C\sqrt{\delta}.
\]
Thus, for every $\delta>0$ and $x\in\R$,
\[
|\Theta(x+\ii\delta)-\Theta(x)|
=\lim_{N\to\infty} |\Theta_N(x+\ii\delta)-\Theta_N(x)|
\leq C\sqrt{\delta}.
\]
Hence the limiting function $\Theta$ defined on $\C_+$ is a holomorphic extension the limiting function $\Theta$ defined on $\R$. We conclude that $\Theta$ is an inner function in the Hardy space $H^\infty(\R)$.

The Beurling factorization theorem  (see for instance~\cite{Duren70book}) now implies that $\Theta$ admits the decomposition
\[
\Theta(z)
	=\ee^{\ii\gamma}\ee^{\ii\alpha z}B(z)\exp\left(\ii\int_{\R}\frac{1+tz}{t-z}\dd \nu(t)\right),
	\quad z\in\C_+,
\]
where:
\begin{itemize}
\item $\gamma\in\R$,  $\alpha\geq 0$;
\item  $B$ is a Blaschke product
\[
B(z)=\prod_{j\geq 1}\left(\ee^{\ii\alpha_j}\frac{z+\overline{p_j}}{z+p_j}\right)
\]
where $\alpha_j=0$ if $p_j=\ii$, and $\alpha_j=\arg(p_j^2+1)$ if $p_j\neq \ii$;
\item $\nu$ is a nondecreasing function of bounded variation over $\R$ satisfying $\nu'(t)=0$ almost-everywhere, so that the measure $d\nu$ is singular with respect to the Lebesgue measure.
\end{itemize}
Taking the logarithmic derivative, we find that for every $x\in\R$,
\begin{equation}\label{eq:u-3sum}
u_0(x)
	=-\ii\frac{\Theta'(x)}{\Theta(x)}
	=\alpha 
	+ \sum_{j\geq 1}\frac{2\Im(p_j)}{|x+p_j|^2}
	+\int_{\R}\frac{(1+t^2)\dd\nu(t)}{(x-t)^2}.
\end{equation}
Given that $u_0$ is a sum of three nonnegative terms and $u_0\in L^2(\R)$, each of these terms should belong to $L^2(\R)$. In particular, we get that $\alpha=0$.  We claim that $\dd\nu$ is identically zero. Indeed, since $x\mapsto \int_{\R}\frac{(1+t^2)\dd\nu(t)}{(x-t)^2}$ and $x\mapsto \frac{1}{1+|x|}$ both belong to $L^2(\R)$, the following integral needs to be finite:
\[
I
	=\int_{\R}(1+t^2)\dd\nu(t)\int_{\R}\frac{\dd x}{(1+|x|)(x-t)^2}.
\]
However, for every $t\in\R$, we know that
\[
\int_{\R}\frac{\dd x}{(1+|x|)(x-t)^2}=+\infty.
\]
We deduce that $\dd\nu\equiv 0$. Going back to~\eqref{eq:u-3sum}, we find that $u_0$ either belongs to $\mathscr M$ or is of the form~\eqref{eq:u0} and belongs to $L^2(\R)$, so that Lemma~\ref{lem:L2} implies that~\eqref{eq:norm} holds. We conclude that $u_0$ satisfies Definition~\ref{def:u0}.
\end{proof}

\bibliographystyle{alpha}
\bibliography{bibliography}

@article{GassotGM26,
  title={{A proof of the soliton resolution conjecture for the Benjamin--Ono equation}},
  author={Gassot, L. and G{\'e}rard, P. and Miller, P. D.},
  journal={arXiv preprint arXiv:2601.10488},
  year={2026}
}

@Article{BlackstoneGGM24a,
  	author = {Blackstone, E. and Gassot, L. and G\'erard, P. and Miller, P. D.},
	title = {The {B}enjamin-{O}no Initial-Value Problem for Rational Data with Application to Long Time Asymptotics and Scattering},
	JOURNAL={{Ann. Inst. H. {P}oincaré C, Anal. Non linéaire}},
    year={2025},
    doi={10.4171/AIHPC/169}
    }

@article{Gerard22,
	author = {Gérard, P.},
	doi = {10.2140/tunis.2023.5.593},
	fjournal = {Tunisian Journal of Mathematics},
	issn = {2576-7658,2576-7666},
	journal = {Tunis. J. Math.},
	mrclass = {35Q53 (35C05 37K15 47B35)},
	mrnumber = {4662323},
	number = {3},
	pages = {593--603},
	title = {An explicit formula for the {B}enjamin-{O}no equation},
	volume = {5},
	year = {2023}
}

@article{Gerard26,
title={{L}ectures on integrable equations of {B}enjamin-{O}no type},
author={G{\'e}rard, P.},
journal={EMS Surv. Math. Sci.},
year={2026},
doi={10.4171/EMSS/111}
}

@article{BadreddineKV25,
author={Badreddine, R. and Killip, R. and Vi\c{s}an, M.},
title={{O}rbital stability of {B}enjamin-{O}no multisolitons},
year={2025},
journal={Arxiv:2509.14153},
 archivePrefix={arXiv},
 primaryClass={math.AP},
url = {https://arXiv.org/abs/2509.14153},
}

@book {KleinS21,
    AUTHOR = {Klein, C. and Saut, J.-C.},
     TITLE = {Nonlinear dispersive equations---inverse scattering and {PDE} methods},
    SERIES = {Applied Mathematical Sciences},
    VOLUME = {209},
    ADDRESS = {Cham},
 PUBLISHER = {Springer},
      YEAR = {2021},
      ISBN = {978-3-030-91426-4; 978-3-030-91427-1},
   MRCLASS = {35-02 (35C08 35Q51 35Q53 35Q55 37K40)},
  MRNUMBER = {4400881},
       DOI = {10.1007/978-3-030-91427-1}
}

@article {Wu16a,
    AUTHOR = {Wu, Y.},
     TITLE = {Simplicity and finiteness of discrete spectrum of the
              {B}enjamin-{O}no scattering operator},
   JOURNAL = {SIAM J. Math. Anal.},
  FJOURNAL = {SIAM Journal on Mathematical Analysis},
    VOLUME = {48},
      YEAR = {2016},
    NUMBER = {2},
     PAGES = {1348--1367},
      ISSN = {0036-1410,1095-7154},
   MRCLASS = {35P25 (35Q53 35R30 47A40)},
  MRNUMBER = {3484397},
       DOI = {10.1137/15M1030649},
       URL = {https://doi.org/10.1137/15M1030649},
}

@article{Matsuno79,
	address = {DIRAC HOUSE, TEMPLE BACK, BRISTOL, ENGLAND BS1 6BE},
	affiliation = {MATSUNO, Y (Corresponding Author), KYOTO UNIV,FAC SCI,DEPT PHYS,KYOTO 606,JAPAN.},
	author = {Matsuno, Y.},
	da = {2024-05-17},
	doc-delivery-number = {GP929},
	doi = {10.1088/0305-4470/12/4/019},
	issn = {0305-4470},
	journal = {J. Phys. A Math. Gen.},
	journal-iso = {J. Phys. A-Math. Gen.},
	language = {English},
	number = {4},
	number-of-cited-references = {7},
	pages = {619-621},
	publisher = {IOP PUBLISHING LTD},
	research-areas = {Physics},
	times-cited = {88},
	title = {EXACT MULTI-SOLITON SOLUTION OF THE {B}ENJAMIN-{O}NO EQUATION},
	type = {Note},
	unique-id = {WOS:A1979GP92900019},
	usage-count-last-180-days = {0},
	usage-count-since-2013 = {5},
	volume = {12},
	web-of-science-categories = {Physics, Multidisciplinary; Physics, Mathematical},
	web-of-science-index = {Science Citation Index Expanded (SCI-EXPANDED)},
	year = {1979}
}

@article {Paulsen24,
    AUTHOR = {Paulsen, M. O.},
     TITLE = {Justification of the {B}enjamin-{O}no equation as an internal
              water waves model},
   JOURNAL = {Ann. PDE},
  FJOURNAL = {Annals of PDE. Journal Dedicated to the Analysis of Problems
              from Physical Sciences},
    VOLUME = {10},
      YEAR = {2024},
    NUMBER = {2},
     PAGES = {Paper No. 25, 129},
      ISSN = {2524-5317,2199-2576},
   MRCLASS = {35Q35 (76B55)},
  MRNUMBER = {4832425},
MRREVIEWER = {Qixiang\ Li},
       DOI = {10.1007/s40818-024-00190-z},
       URL = {https://doi-org.ezproxy.universite-paris-saclay.fr/10.1007/s40818-024-00190-z},
}

@Misc{amsmath,
  author =	 {{American Mathematical Society}},
  title =	 {User's Guide for the \texttt{amsmath} Package
                  (Version 2.0)},
  url =		 {ftp://ftp.ams.org/pub/tex/doc/amsmath/amsldoc.pdf},
  urldate =	 {2015-07-30},
  year =	 2002}

@article{AmickToland1991,
	author = {Amick, C. J. and Toland, J. F.},
	doi = {10.1007/BF02392447},
	journal = {Acta Mathematica},
	number = {1},
	pages = {107--126},
	publisher = {Springer},
	title = {{Uniqueness and related analytic properties for the Benjamin-Ono equation---a nonlinear Neumann problem in the plane}},
	volume = {167},
	year = {1991}
}

@article{Benjamin67,
	address = {40 WEST 20TH STREET, NEW YORK, NY 10011-4211},
	author = {Benjamin, T. B.},
	da = {2024-05-17},
	doc-delivery-number = {99234},
	doi = {10.1017/S002211206700103X},
	issn = {0022-1120},
	journal = {J. Fluid Mech.},
	journal-iso = {J. Fluid Mech.},
	language = {English},
	number = {3},
	number-of-cited-references = {20},
	pages = {559-592},
	publisher = {CAMBRIDGE UNIV PRESS},
	research-areas = {Mechanics; Physics},
	times-cited = {843},
	title = {INTERNAL WAVES OF PERMANENT FORM IN FLUIDS OF GREAT DEPTH},
	type = {Article},
	unique-id = {WOS:A19679923400010},
	usage-count-last-180-days = {3},
	usage-count-since-2013 = {34},
	volume = {29},
	web-of-science-categories = {Mechanics; Physics, Fluids \& Plasmas},
	web-of-science-index = {Science Citation Index Expanded (SCI-EXPANDED)},
	year = {1967}
}

@article {KillipLaurensVisan23-BO,
    AUTHOR = {Killip, R. and Laurens, T. and Vi\c{s}an, M.},
     TITLE = {Sharp well-posedness for the {B}enjamin-{O}no equation},
   JOURNAL = {Invent. Math.},
  FJOURNAL = {Inventiones Mathematicae},
    VOLUME = {236},
      YEAR = {2024},
    NUMBER = {3},
     PAGES = {999--1054},
      ISSN = {0020-9910,1432-1297},
   MRCLASS = {35C08},
  MRNUMBER = {4743514},
       URL = {https://doi.org/10.1007/s00222-024-01250-8},
}

@article{Sun2020,
 title={{Complete integrability of the Benjamin--Ono equation on the multi-soliton manifolds}},
  author={Sun, R.},
DOI={10.1007/s00220-021-03996-1},
url={https://doi.org/10.1007/s00220-021-03996-1},
volume={383},
journal={Communications in Mathematical Physics},
publisher={Springer},
year={2021},
month={4},
pages={1051--1092}
}

@incollection{Duren70book,
title = "Chapter 11 - {H}$^p$ Spaces Over A {H}alf-{P}lane",
author = "P. L. Duren",
series = "Pure and Applied Mathematics",
address = "London",
publisher = "Elsevier, Academic Press",
volume = "38",
pages = "187-199",
year = "1970",
booktitle = "Theory of H$^p$ Spaces",
issn = "0079-8169",
doi = "https://doi.org/10.1016/S0079-8169(08)62675-6"
}

@article{GerardHe26hierarchy,
  title={{An Explicit Formula for the Benjamin-Ono Hierarchy with Applications to Traveling Waves and Zero-Dispersion Limits}},
  author={Patrick Gérard and Jiao He},
  year={2026},
  journal={Arxiv:2604.20464},
   primaryClass={math.AP},
}

\end{document}